\documentclass[12pt]{article}
\usepackage{amsfonts}
\usepackage{amsmath,amsthm}
\usepackage{cases}

\newcommand{\Hess}{{\rm Hess}}
\newcommand{\Ric}{{\rm Ric}}
\newtheorem{lemma}{Lemma}[section]
\newtheorem{theorem}{Theorem}[section]
\newtheorem{proposition}[theorem]{Proposition}
\theoremstyle{definition}
\newtheorem{definition}{Definition}[section]
\newtheorem{corollary}[theorem]{Corollary}
\newtheorem{remark}{Remark}

\newcommand{\comment}[1]{}
\newcommand{\R}{\mathbb{R}}

\numberwithin{equation}{section}
\begin{document}
\title{Equivalent Properties of CD Inequality on Graph}
\author{Yong Lin\footnotemark[1], Shuang Liu}
\date{}
\maketitle

\renewcommand{\thefootnote}{\fnsymbol{footnote}}
\footnotetext[1]{Supported by the National Natural Science Foundation of China(GrantNo.$11271011$), and supported by the Fundamental Research Funds for the Central Universities and the Research Funds of Renmin University of China($11$XNI$004$).}

\begin{center}
\textbf{Abstract}
\end{center}
We study some equivalent properties of the curvature-dimension conditions $CD(n,K)$ inequality on infinite, but locally finite graph. These equivalences are gradient estimate, Poincar\'{e} type inequalities and reverse Poincar\'{e} inequalities. And we also obtain one equivalent property of gradient estimate for a new notion of curvature-dimension conditions $CDE'(\infty, K)$ at the same assumption of graphs.\\


\section{Introduction}
\subsection{Preliminaries}
Let $G=(V,E)$ be a finite or infinite graph. We allow the edges on the graph to be weighted, we consider a symmetric weight function $\omega: V\times V\rightarrow [0,\infty)$, the edge $xy$ from $x$ to $y$ has weight $\omega_{xy}>0$. In this paper, we assume this weight function is symmetric($\omega_{xy}=\omega_{yx}$). Moreover we assume the graph is connected, which implies the weight function satisfies
$$\omega_{\min}=\inf_{x,y \in V}\omega_{xy}>0.$$
And the graph we are interested is locally finite,
$$m(x):=\sum_{y\sim x}\omega_{xy}<\infty, \quad \forall x\in V.$$

Given a positive and finite measure $\mu:V \rightarrow \mathbb{R^{+}}$ on graph. We denote by $V^{\mathbb{R}}$ the space of real functions on $V$. Let $\ell^{p}(V,\mu)=\{f \in V^{\mathbb{R}}:\sum_{x\in V}\mu(x)|f(x)|^{p}<\infty\}, 1\leq p< \infty$, be
the space of $\ell^{p}$ integrable functions on $V$ with respect to the measure $\mu$. If for any $f,g \in \ell^{2}(V,\mu)$, let the inner product as $\langle f,g\rangle=\sum_{x\in V}\mu(x)f(x)g(x)$, then the space of $\ell^{2}(V,\mu)$ is a Hilbert space. For $p=\infty$, let $\ell^{\infty}(V,\mu)=\{f \in V^{\mathbb{R}}:\sup_{x\in V}|f(x)|<\infty\}$ be the set of bounded functions. For every function $f\in \ell^{p}(V,\mu),1\leq p\leq \infty$, defining the norm by
\[\|f\|_p=\left(\sum_{x\in V}\mu(x)|f(x)|^p\right)^{\frac{1}{p}}, 1\leq p<\infty ~\mbox{and}~ \|f\|_{\infty}=\sup_{x\in V}|f(x)|.\]

The $\mu-$Laplacian $\Delta:V^{\mathbb{R}}\rightarrow V^{\mathbb{R}}$ on $G$ is the linear operator, defined by, for any $x\in V$,
$$\Delta f(x)=\frac{1}{\mu(x)}\sum_{y\sim x} \omega_{xy}(f(y)-f(x)).$$
It will be useful to introduce an abbreviated notation for "averaged sum",
$$\widetilde{\sum_{y\sim x}}h(y)=\frac{1}{\mu(x)}\sum_{y\sim x} \omega_{xy}h(y) \quad \forall x\in V.$$
If $f\in \ell^{\infty}(V,\mu)$, under the assumption of locally finite, it is known immediately that for any $x\in V$, $\Delta f(x)$ is the sum of finite terms. The two most natural choices are the case where $\mu(x)=m(x)$ for all $x\in V$, which is the normalized graph Laplacian, and the case $\mu$ $\equiv 1$ which is the standard graph Laplacian. Furthermore, in this paper we assume
$$D_{\mu}:=\max_{X \in V}\frac{m(x)}{\mu(x)}<\infty.$$

\subsection{Curvature-dimension inequalities}

In this subsection we introduce the notion of the curvature-dimension inequalities on graphs. A graph is a metric space with a proper distance, such as its natural graph distance. Metric spaces play a important role in many fields of mathematics. If admitting all kinds of singularities, metric spaces constitute natural generalizations of manifolds, and they provide rich geometric structures too. We regard graphs as discretizing Riemannian manifolds. For many fundamental results in geometric analysis on metric space, the crucial ingredients are bounds for the Ricci curvature of the underlying manifolds as pinoeered by Barkry and Emery [BE].

On a Riemannian manifold, Bochner's identity is given by
$$\frac{1}{2}\Delta |\nabla f|^2 = \langle \nabla f, \nabla \Delta f \rangle + \|\Hess f\|_2^2 +  \Ric(\nabla f, \nabla f),$$
which is establishes an important connection between Ricci curvature and analytic properties of a manifold. An immediate consequence of the Bochner's identity is that on an $n$-dimensional manifold whose Ricci curvature is bounded from below by $K$ one has
\begin{equation*}
\frac{1}{2}\Delta |\nabla f|^2 \geq \langle \nabla f, \nabla \Delta f \rangle + \frac{1}{n}(\Delta f)^2 + K |\nabla f|^2,
\end{equation*}
which is called the curvature-dimension inequality. It can be used as a substitute for the lower Ricci curvature bound on metric spaces by Bakry and Emery from [BE].

Bakry and Ledoux [BL] even manage to generalize the above curvature-dimension inequality to Markov operators on general measure spaces when the operator satisfies a chain rule type formula. Unfortunately such a formula cannot hold in a discrete setting. In fact, only the function $u^2$ and $u^\frac{1}{2}$ satisfy the chain rule from [LY$10$] and [BHL+] on graphs, which is probably one of the most crucial observation to start researches in this field.

First we need to recall the definition of two bilinear forms associated to the $\mu-$Laplacian. These notations are from the paper [LY10] and [BHL+].
\begin{definition}
The gradient form $\Gamma$ is defined by
\begin{equation}
\begin{split}
2\Gamma(f,g)(x)& =(\Delta(f\cdot g)-f\cdot \Delta(g)-\Delta(f)\cdot g)(x)\\
               & =\frac{1}{\mu(x)}\sum_{y\sim x}\omega_{xy}(f(y)-f(x))(g(y)-g(x)).
\end{split}
\notag
\end{equation}
We write $\Gamma(f)=\Gamma(f,f)$.
\end{definition}
Similarly,
\begin{definition}
The iterated gradient form $\Gamma_{2}$ is defined by
$$2\Gamma_{2}(f,g) = \Delta\Gamma(f,g)-\Gamma(f,\Delta g)-\Gamma(\Delta f,g).$$
We write $\Gamma_{2}(f)=\Gamma_{2}(f,f)$.
\end{definition}
\begin{definition}
The graph $G$ satisfies the CD inequality $CD(n,K)$ if, for any function $f$
$$\Gamma_{2}(f)\geq \frac{1}{n}(\Delta f)^{2}+K\Gamma(f).$$
\end{definition}

\begin{definition}
We say that a graph $G$ satisfies the $CDE'(x,n,K)$, if for any positive function $f : V\to \R^+$, we have
$$\widetilde{\Gamma_2}(f)(x) \geq \frac{1}{n} f(x)^2\left(\Delta \log f\right)(x)^2 + K \Gamma(f)(x).$$
We say that $CDE'(n,K)$ is satisfied if $CDE'(x,n,K)$ is satisfied for all $x \in V$.
\end{definition}

\section{Heat semigroup on graph}
In this section we are interested in the positive solution of the heat equation
$$\Delta u=\partial_{t}u$$
on graphs. And we focus on the heat kernel $p_{t}(x,y)$, a fundamental solution of the heat equation defined by, if for any bounded initial condition $u_{0}:V\rightarrow \mathbb{R}$, the function
$$u(t,x)=\sum_{y\in V} \mu(y)p_{t}(x,y)u_{0}(y)\quad t>0,x\in V$$
is differentiable in $t$, satisfies the heat equation, and if for any $x\in V$,
$\lim_{t\rightarrow 0^{+}}u(t,x)=u_{0}(x)$
holds.

Follow the paper [WR], we   definite the heat kernel for the above $\Delta$ we introduced on the infinite but locally finite graph  $G=(V,E)$. For any subset $U \subset V$ denotes always a finite subset, we denote by $\overset{\circ}{U}=\{x\in U:y\sim x,y\in U\}$ the interior of $U$. The boundary of $U$ is $\partial U = U\setminus \overset{\circ}{U}$. We consider the Dirichlet problem (DP),
\[\left\{
  \begin{array}{ll}
   \partial_{t}u(t,x)-\Delta_{U}u(t,x)=0, & \hbox{$x\in \overset{\circ}{U},t>0$,} \\
    u(0,x)=u_{0}(x), & \hbox{ $x\in \overset{\circ}{U}$,} \\
    u\mid_{[0,\infty)\times \partial U}=0.
  \end{array}
\right.\]
where $\Delta_{U}: \ell ^{2} (\overset{\circ}{U},\mu)\rightarrow \ell ^{2}(\overset{\circ}{U},\mu)$ denotes the Dirichlet Laplacian on $\overset{\circ}{U}$.

The operator $\Delta_{U}$ is a generator of the heat semigroup $P_{t,U}=e^{t\Delta_{U}}$,$t>0$. According to spectral graph theory, we can find the easy knowing, $e^{t\Delta_{U}}\phi_{i}=e^{-t\lambda_{i}}\phi_{i}$. We can define the heat kernel $p_{U}(t,x,y)$ for the finite subset $U$ by
$$p_{U}(t,x,y)=P_{t,U}\delta_{y}(x),\quad \forall x,y\in \overset{\circ}{U}$$
where $\delta_{y}(x)=\sum_{i=1}^{n}\langle\Phi_{i},\delta_{y}\rangle\Phi_{i}(x)=\sum_{i=1}^{n}\Phi_{i}(x)\Phi_{i}(y)$.
It is easy to know the heat kernel satisfies
$$ p_{U}(t,x,y)=\sum_{i=1}^{n}e^{-\lambda_{i}t}\phi_{i}(x)\phi_{i}(y),\quad \forall x,y\in \overset{\circ}{U}.$$

Let $U\subset V$, $k \in \mathbb{N}$ be a sequence of finite subsets with $U_{k}\subset \overset{\circ}{U}_{k+1}$ and $\cup_{k \in \mathbb{N}}U_{k}=V$. Such a sequence always exists and can be constructed as a sequence $U_{k} = B_{k}(x_{0})$ of metric balls with center $x_{0}\in V$ and radius $k$. The connectedness of our graph $G$ implies that the union of these $U_k$ equals $V$. In the following, we will write $p_{k}$ for the heat kernel $p_{U_{k}}$ on $U_{k}$, and define $p_{k}(t,x,y)$ as a function on $(0,\infty)\times V\times V$ by,
\[p_{k}(t,x,y)=
\left\{
  \begin{array}{ll}
     p_{U_{k}}(t,x,y), & \hbox{$x,y\in \overset{\circ}{U_{k}}$;} \\
     0, & \hbox{o.w.}
  \end{array}
\right.
\]
For any $t>0,x,y\in V,$ we let
$$p_{t}(x,y)=\lim_{k\rightarrow \infty}p_{k}(t,x,y),$$
$p_{t}(x,y)$ is the heat kernel what we want and does not depend on the choice of the exhaustion sequence $U_{k}$. For completeness, we conclude all properties we will use in this paper of the heat kernel $p_t(x,y)$ as follows.
\begin{remark}
For $t,s>0$, $\forall x,y\in V$, we have
\begin{enumerate}
  \item $p_t(x,y)=p_t(y,x)$
  \item $p_t(x,y)\geq 0$,
  \item $\sum_{y\in V}\mu(y)p_t(x,y) \leq 1$,
  \item $\lim_{t\rightarrow 0^{+}}\sum_{y\in V}\mu(y)p_t(x,y) = 1$,
  \item $\partial_{t}p_t(x,y)=\Delta_yp_t(x,y)=\Delta_xp_t(x,y)$
  \item $\sum_{z\in V}\mu(z)p_t(x,z)p_s(z,y)=p_{t+s}(x,y)$
\end{enumerate}
\end{remark}
So far now, we can obtain some properties of the operator $P_{t}$ defined by, for any bounded function $f\in \ell^{\infty}(V,\mu)$,
$$P_{t}f(x)=\lim_{k\rightarrow \infty}\sum_{y\in V}\mu(y)p_{k}(t,x,y)f(y)=\sum_{y\in V}\mu(y)p_{t}(x,y)f(y).$$
\begin{proposition}
For any bounded function $f,g\in\ell^{\infty}(V,\mu)$, and $t,s>0$, for any $x\in V$,
\begin{enumerate}
  \item  $P_{t}$ is a bounded operator and a contraction,
  \item $P_{t}\circ P_{s}f(x)=P_{t+s}f(x)$,
  \item $\Delta P_{t}f(x)=P_{t} \Delta f(x)$.
\end{enumerate}
\end{proposition}
\begin{proof}
The first one immediately comes from the definition of $P_{t}f$ and item 3 in Remark 1.

For any bounded function $f\in \ell^{\infty}(V,\mu)$, and any $x\in V$,
notice $\lim_{k\rightarrow \infty} p_{k}(t,x,y)$ does not depend on the choice of the exhaustion sequence $U_{k}$, so
\[
\begin{split}
P_{t}\circ P_{s}f(x)&=\lim_{k\rightarrow \infty}\sum_{y\in V}\mu(y)p_{k}(t,x,y)\sum_{z\in V}\mu(z)p_{k}(s,y,z)f(z)\\
                &=\lim_{k\rightarrow \infty}\sum_{z\in V}\mu(z)\left(\sum_{y\in V}\mu(y)p_{k}(t,x,y)p_{k}(s,y,z)\right)f(z)\\
                &=\lim_{k\rightarrow \infty}\sum_{z\in V}\mu(z)p_{k}(t+s,x,z)f(z)\\
                &=P_{t+s}f(x).
\end{split}
\]

Notice the function $f$ is bounded,  there exists a constant $C>0$, such that for any $x\in V$, $\sup_{x\in V}|f(x)|\leq C$, we have
\[\sum_{y\in V}\sum_{z\sim y}|\omega_{yz} p_{t}(x,y)f(z)|\leq D_{\mu}C\sum_{y\in V}\mu(y)p_{t}(x,y)\leq D_{\mu}C<\infty,\]
and
\[\sum_{y\in V}\sum_{z\sim y}|-\omega_{yz} p_{t}(x,y)f(y)|<\infty.\]
Then,
\[
\begin{split}
\Delta P_{t}f(x)&=\Delta_{x} \left(\sum_{y\in V}\mu(y)p_{t}(x,y)f(y)\right)\\
                &=\sum_{y\in V}\mu(y)\Delta_{y} p_{t}(x,y)f(y)\\
                &=\sum_{y\in V}\sum_{z\sim y}\omega_{yz} (p_{t}(x,z)-p_{t}(x,y))f(y)\\
                &=\sum_{y\in V}\sum_{z\sim y}\omega_{yz} p_{t}(x,z)f(y)-\sum_{y\in V}\sum_{z\sim y}\omega_{yz} p_{t}(x,y)f(y)\\
                &=\sum_{y\in V}\sum_{z\sim y}\omega_{yz} p_{t}(x,z)f(y)-\sum_{y\in V}\sum_{z\sim y}\omega_{yz} p_{t}(x,z)f(z)\\
                &=\sum_{y\in V}\sum_{z\sim y}\omega_{yz} p_{t}(x,z)(f(y)-f(z))\\
                &=P_{t}\Delta f(x).
\end{split}
\]
This ends the proof of Proposition 3.1.
\end{proof}
We need to clear up that the operator $P_t$ above is the heat kernel of the heat semigroup $e^{t\Delta}$.
\begin{proposition}
For any bounded function $f,g\in\ell^{\infty}(V,\mu)$, and $t>0$, for any $x\in V$,
\[P_t f(x)=e^{t\Delta}f(x).\]
\begin{proof}
Consider the function
\[v(t,x)=P_tf(x)-e^{t\Delta}f(x),\]
we are going to show that $v=0$.
\[\begin{split}
\sum_{x\in V}\mu(x)v^2(t,x)
&=\sum_{x\in V}\mu(x)\int_0^t\partial_\tau v^2(\tau,x)d\tau\\
&=2\sum_{x\in V}\mu(x)\int_0^t v(\tau,x)\Delta v(\tau,x)d\tau\\
&=2\int_0^t\sum_{x\in V}\mu(x) v(\tau,x)\Delta v(\tau,x)d\tau\\
&=-2\int_0^t\sum_{x\in V}\mu(x) \Gamma(v(\tau,\cdot))(x)d\tau\leq 0,
\end{split}
\]
it follows $v=0$. The interchange of summation and integration in the calculation is justified by Tonelli's Theorem as the iterated integral are finite since $P_t$ and  $e^{t\Delta}$ are contractions.
\end{proof}
\end{proposition}

\section{Main results}
In many papers, such as [BL] and [W], they have proved equivalent semigroup properties for curvature-dimension condition on Riemannian manifold, if assuming the semigroup be a diffusion semigroup, that is, the operator $\Delta$ satisfies the chain role. But as what this paper says before, it doesn't hold in discrete setting. In this section, we will introduce the similar results on graphs at the assumption of $CD(n,K)$. Moreover, we obtain one of the equivalent semigroup properties, the gradient bound, if the graph satisfies the condition $CDE'(\infty,K)$.

For a connected and locally finite graph $G=(V,E)$, for any $0\leq s < t, x\in V$, any positive and bounded function $0<f\in \ell^{\infty}(V,\mu)$, we know $P_{t-s} f$, $\Gamma(P_{t-s} f)$ and $\Gamma(\sqrt{P_{t-s} f})$ are all bounded and their boundaries are not related with $s$, because of the boundness of the operators $P_t$ and $\Gamma$.

We also need the following Lemma from [HLLY] to clarify that
$\Gamma(\sqrt{f},\frac{\Delta P_{t}f}{2\sqrt{P_{t}f}})$ is bounded if $f$ is bounded.
\begin{lemma}
For any positive and bounded solution $0<u\in \ell^{\infty}(V,\mu)$ to the heat equation on $G$, if the graph satisfies the condition $CDE'(n,K)$, then the function
$\frac{\Delta u}{2\sqrt{u}}$ on $G$ is bounded.
\end{lemma}

Consider the following three functions,
$\phi(s,x)=P_{s}(P_{t-s} f)^{2}(x),$
$\varphi(s,x)=P_{s}\left(\Gamma(P_{t-s}f)\right)(x),$
$\psi(s,x)=P_{t}(\Gamma(\sqrt{P_{t-s} f}))(x).$
\begin{lemma}
For any $0\leq s < t$, $x\in V$, the following assertions are true,
$$\partial_{s}\phi(s,x)=2P_{s}\left(\Gamma(P_{t-s} f)\right)(x),$$
$$\partial_{s}\varphi(s,x)=2P_{s}(\Gamma_{2}(P_{t-s}f))(x),$$
$$\partial_{s}\psi(s,x)=2 P_{s}(\widetilde{\Gamma}_{2}(\sqrt{P_{t-s} f}))(x).$$
\end{lemma}

\begin{proof}
For any $0\leq s < t$, $x\in V$,
\[
\begin{split}
\partial_{s}\phi(x,s)
&=\partial_{s}\sum_{y\in V}\mu(y)p_{s}(x,y)(P_{t-s}f)^{2}(y)\\
&=\lim_{k\rightarrow \infty}\partial_{s}\sum_{y\in V}\mu(y)p_{k}(s,x,y)(P_{t-s}f)^{2}(y)\\
&=\lim_{k\rightarrow \infty}\sum_{y\in V}\mu(y)\partial_{s}\left(p_{k}(s,x,y)(P_{t-s}f)^{2}(y)\right)\\
&=\lim_{k\rightarrow \infty}\sum_{y\in V}\mu(y)p_{k}(s,x,y)\left(\Delta(P_{t-s}f)^{2}(y)-2P_{t-s}f\Delta P_{t-s}f(y)\right)\\
&=P_{s}\left(\Delta(P_{t-s} f)^{2}-2\Delta(P_{t-s} f)P_{t-s} f\right)(x)\\
&=2P_{s}(\Gamma(P_{t-s} f))(x).
\end{split}
\]
For any $t>0$, and any positive function $f\in \ell^{\infty}(V,\mu)$, the interchange of limitation and deviation in the second step of the above calculation is for the convergence of summation is uniform with respect of $s$ on compact subsets of $(0,\infty)$. The interchange of summation and deviation in the third step of the above calculation is because $p_{k}(s,x,y)$ is non-zero only for finitely many items. And in the forth equality,
\[
\begin{split}
&\sum_{y\in V}\mu(y)\partial_{s}\left(p_{k}(s,x,y)(P_{t-s}f)^{2}(y)\right)\\
&=\sum_{y\in V}\mu(y)\Delta p_{k}(s,x,y)(P_{t-s}f)^{2}(y)-\sum_{y\in V}\mu(y) 2p_{k}(s,x,y)P_{t-s}f(y)\Delta P_{t-s}f(y)\\
&=\sum_{y\in V}\mu(y) p_{k}(s,x,y)\Delta(P_{t-s}f)^{2}(y)-\sum_{y\in V}\mu(y) 2p_{k}(s,x,y)P_{t-s}f(y)\Delta P_{t-s}f(y)\\
&=\sum_{y\in V}\mu(y) p_{k}(s,x,y)\left(\Delta(P_{t-s}f)^{2}(y)- 2P_{t-s}f(y)\Delta P_{t-s}f(y)\right),
\end{split}
\]

We can summarise from the above proof, for any positive and bounded functional $f_s\in \ell^{\infty}(V,\mu)$, we have
$$\partial_{s}P_{s}(f_{s})(x)=P_{s}(\Delta f_{s}+\partial_{s} f_{s})(x).$$

From that, we can simply obtain the following results,
\[
\begin{split}
\partial_{s}\varphi(s,x)
& =\partial_{s}P_{s}\left(\Gamma(P_{t-s}f)\right)(x)\\
& =P_{s}\left(\Delta\Gamma(P_{t-s}f)+ \partial_{s}\Gamma(P_{t-s}f\right)(x)\\
& =P_{s}\left(\Delta\Gamma(P_{t-s}f)- 2\Gamma(P_{t-s}f,\Delta P_{t-s}f\right)(x)\\
& =P_{s}\left(\Gamma_{2}(P_{t-s}f)\right)(x),
\end{split}
\]
where,
\[
\begin{split}
\partial_{s}\Gamma(P_{t-s} f)(x)
&= \frac{1}{2}\partial_{s}\widetilde{\sum_{y\sim x}}\left(P_{t-s} f(y)-P_{t-s} f(x)\right)^{2}\\
&= \widetilde{\sum_{y\sim x}}(P_{t-s} f(y)-P_{t-s} f(x)) (\partial_{s}P_{t-s} f(y)-\partial_{s}P_{t-s} f(x))\\
&= 2 \Gamma(P_{t-s} f,\partial_{s}P_{t-s} f)(x)\\
&= -2 \Gamma(P_{t-s} f,\Delta P_{t-s} f)(x).
\end{split}
\]

Furthermore,
\[
\begin{split}
\partial_{s}\psi(s,x)
& =\partial_{s}P_{s}(\Gamma(\sqrt{P_{t-s} f}))(x)\\
& =P_{s}\left( \Delta \Gamma(\sqrt{P_{t-s} f}) + \partial_{s}\Gamma(\sqrt{P_{t-s} f})\right)(x)\\
& =P_{s}\left( \Delta \Gamma(\sqrt{P_{t-s} f}) - 2\Gamma(\sqrt{P_{t-s} f},\frac{\Delta P_{t-s} f}{2\sqrt{P_{t-s} f}})\right)(x)\\
& =2 P_{s}(\widetilde{\Gamma}_{2}(\sqrt{P_{t-s} f}))(x).
\end{split}
\]
That ends the proof.
\end{proof}
\subsection{The equivalent properties of $CD$ condition}
In this subsection, we introduce one of the main results in this paper, including the gradient estimate, Poincar\'{e} inequalities and reverse
Poincar\'{e} inequalities at the assumption of  $CD$ inequality. The proof is close to Feng-Yu Wang from [W]. Recently, the following gradient estimate of (1) for finite graphs by Liu-
Peyerimhoff [LP], and for unbounded Laplace operator by Hua-Lin [HL], has been proved for $n=\infty$.

\begin{theorem}
For any $K \in \mathbb{R}$, $t\geq 0$, and any positive and bounded function $0<f\in \ell^{\infty}(V,\mu)$, the condition $CD(n,-K)$ is equivalent to each of the following statements:
\begin{enumerate}
  \item $\Gamma(P_{t}f) \leq e^{2Kt}P_{t}(\Gamma(f))-\frac{2}{n}\int_{0}^{t}e^{2Ks}P_{s}\left(P_{t-s} \Delta f\right)^{2} ds.$ \label{equ:1}
  \item $\Gamma(P_{t}f) \leq e^{2Kt}P_{t}(\Gamma(f))-\frac{e^{2Kt}-1}{Kn}\left(P_{t} \Delta f\right)^{2}.$ \label{equ:2}
  \item $P_{t}f^{2}-(P_{t}f)^{2}\leq \frac{e^{2Kt}-1}{K}P_{t}(\Gamma(f))-\frac{e^{2Kt}-1-2Kt}{K^{2}n}\left(P_{t} \Delta f\right)^{2}.$ \label{equ:3}
  \item $P_{t}f^{2}-(P_{t}f)^{2}\geq \frac{1-e^{2Kt}}{K}\Gamma(P_{t}f)+\frac{e^{-2Kt}-1+2Kt}{K^{2}n}\left(P_{t} \Delta f\right)^{2}.$ \label{equ:4}
\end{enumerate}
\end{theorem}

\begin{proof}
First, we prove the condition $CD(n,-K)$ implies the item \eqref{equ:1}. For $0\leq s\leq t$, consider this functional
$$e^{2Ks}\varphi(s)=e^{2Ks}P_{s}\left(\Gamma(P_{t-s}f)\right),$$
from Lemma 3.2, if the graph satisfies $CD(n,-K)$, then we have
\begin{equation}
\begin{split}
\partial_{s}e^{2Ks}\varphi(s)
& =2e^{2Ks}P_{s}\left(\Gamma_{2}(P_{t-s}f)+K\Gamma(P_{t-s}f)\right)\\
& \geq 2e^{2Ks}P_{s}\left(\frac{1}{n}(\Delta P_{t-s}f)^{2}\right)\\
& =\frac{2e^{2Ks}}{n}P_{s}(\Delta P_{t-s}f)^{2},
\end{split}
\notag
\end{equation}
integrating the above inequality from $0$ to $t$  with respect to $s$, the left side of the above inequality is equal to $e^{2Kt}P_{t}(\Gamma(f))-\Gamma(P_{t}f)$, from Proposition 2.1, we know the fact $\Delta P_{t}=P_{t}\Delta$, then we have
$$\Gamma(P_{t}f)\leq e^{2Kt}P_{t}(\Gamma(f))-\frac{2}{n}\int_{0}^{t}e^{2Ks}P_{s}(P_{t-s}\Delta f)^{2} ds.$$

Then, we introduce the proof of item \eqref{equ:1} implies \eqref{equ:2}. $P_{s}$ satisfies the Cauchy-Schwartz inequality, and from Proposition 2.1, we obtain
\begin{equation}
P_{s}(P_{t-s}\Delta f)^{2} \geq \left(P_{s}(P_{t-s}\Delta f)\right)^{2}= (P_{t}\Delta f)^{2},
\end{equation}\label{equ:5}
so,
\[
\begin{split}
\Gamma(P_{t}f)
&\leq e^{2Kt}P_{t}(\Gamma(f))-\frac{2}{n}(P_{t}\Delta f)^{2}\int_{0}^{t}e^{2Ks} ds\\
&=e^{2Kt}P_{t}(\Gamma(f))-\frac{e^{2Kt}-1}{Kn}\left(P_{t} \Delta f\right)^{2}.
\end{split}
\]

Furthermore, we prove \eqref{equ:2} implies \eqref{equ:3} and \eqref{equ:4}.
From Lemma 3.2, we have
$$\partial_{s}\phi(s)=2P_{s}\left(\Gamma(P_{t-s} f)\right),$$
integrating the above equality from $0$ to $t$, the left side of the above inequality is equal to,
$$\phi(t)-\phi(0)=P_{t}f^{2}-(P_{t}f)^{2}.$$
the right side is not that trivial, from the properties of $P_{t}$ in Proposition 2.1, item \eqref{equ:2} and the inequality \eqref{equ:5}, we have
\begin{equation}
\begin{split}
2P_{s}\left(\Gamma(P_{t-s} f)\right)
& \leq 2P_{s}\left(e^{2K(t-s)}P_{t-s}(\Gamma(f))-\frac{e^{2K(t-s)}-1}{Kn}\left(P_{t-s} \Delta f\right)^{2} \right)\\
& =2 e^{2K(t-s)}P_{t}(\Gamma(f))-2\frac{e^{2K(t-s)}-1}{Kn}P_{s}\left(P_{t-s} \Delta f\right)^{2} \\
& \leq 2 e^{2K(t-s)}P_{t}(\Gamma(f))-2\frac{e^{2K(t-s)}-1}{Kn}\left(P_{t} \Delta f\right)^{2},
\end{split}
\notag
\end{equation}
integrating this equality from $0$ to $t$, we obtain
$$P_{t}f^{2}-(P_{t}f)^{2}\leq \frac{e^{2Kt}-1}{K}P_{t}(\Gamma(f))-\frac{e^{2Kt}-1-2Kt}{K^{2}n}\left(P_{t} \Delta f\right)^{2},$$
that implies that \eqref{equ:3} is true.

On the other side,
\begin{equation}
\begin{split}
2P_{s}\left(\Gamma(P_{t-s} f)\right)
& \geq \frac{2}{e^{2Ks}}\Gamma(P_{t}f)+\frac{2}{e^{2Ks}}\frac{e^{2Ks}-1}{Kn}\left(P_{s} \Delta P_{t-s}f\right)^{2}\\
& =2e^{-2Ks}\Gamma(P_{t}f)+2\frac{1-e^{-2Ks}}{Kn}(P_{t} \Delta f)^{2},
\end{split}
\notag
\end{equation}
integrating this equality from $0$ to $t$, we obtain
$$P_{t}f^{2}-(P_{t}f)^{2}\geq \frac{1-e^{2Kt}}{K}\Gamma(P_{t}f)+\frac{e^{-2Kt}-1+2Kt}{K^{2}n}\left(P_{t} \Delta f\right)^{2},$$
then item \eqref{equ:4} is true too.

Finally, we give the proof that \eqref{equ:3} or \eqref{equ:4} imply the condition $CD(n,-K)$.
Form Proposition 2.2, we have
$$P_{t}=e^{t\Delta}=\sum_{k=0}^{\infty}\frac{t^{k}\Delta^{k}}{k!}.$$
From that, we obtain
$$P_{t}f^{2}=f^{2}+t\Delta f^{2}+\frac{t^{2}}{2}\Delta^{2} f^{2}+o(t^{2}),$$
$$(P_{t}f)^{2}=\left(f+t\Delta f+\frac{t^{2}}{2}\Delta^{2} f+o(t^{2})\right)^{2}=f^{2}+t^{2}(\Delta f)^{2}++2tf\Delta f+t^{2}f\Delta^{2} f +o(t^{2}),$$
so,
\begin{equation*}
\begin{split}
P_{t}f^{2}-(P_{t}f)^{2}
& =2t\Gamma(f)+t^{2}\left(\frac{1}{2}\Delta^{2} f^{2}-(\Delta f)^{2}-f\Delta^{2} f\right)+o(t^{2})\\
& =2t\Gamma(f)+t^{2}\left[\left(\frac{1}{2}\Delta^{2} f^{2}-\Delta(f\Delta f)\right)+\left(\Delta(f\Delta f)-(\Delta f)^{2}-f\Delta^{2} f\right)\right]+o(t^{2})\\
& =2t\Gamma(f)+t^{2}\left(\Delta\Gamma(f)+2\Gamma(f,\Delta f)\right)+o(t^{2}).
\end{split}
\end{equation*}
On the other side,
\begin{equation*}
\begin{split}
\frac{e^{2Kt}-1}{K}P_{t}(\Gamma(f))
& =(2t+2Kt^{2}++o(t^{2}))\cdot (\Gamma(f)+t\Delta\Gamma(f)++o(t))\\
& =2t\Gamma(f)+2t^{2}(\Delta\Gamma(f)+K\Gamma(f))+o(t^{2}),
\end{split}
\end{equation*}
and,
\begin{equation*}
\frac{e^{2Kt}-1-2Kt}{K^{2}n}(P_{t}\Delta f)^{2}=\frac{2 t^{2}}{n}(\Delta f)^{2}+o(t^{2}).
\end{equation*}

From item (3) and the above three equalities, we can obtain,
$$2t\Gamma(f)+t^{2}\left(\Delta\Gamma(f)+2\Gamma(f,\Delta f)\right)-2t\Gamma(f)-2t^{2}(\Delta\Gamma(f)+K\Gamma(f))+\frac{2}{n}t^{2}(\Delta f)^{2}+o(t^{2})\leq 0,$$
from some simple computation, we have
$$t^{2}\left(-2\Gamma_{2}(f)-2K\Gamma(f)+2\frac{1}{n}(\Delta f)^{2}\right)+o(t^{2})\leq 0,$$
let the both side of the above inequality divide by $t^{2}$, and then $t\rightarrow 0$, then
we obtain $CD(n,-K)$ inequality.

Similarly, item (4) implies $CD(n,-K)$ inequality too. Due to $\Gamma$ is bilinear, we have
$$\Gamma(P_{t}f)=\Gamma(f+t\Delta f+o(t))=\Gamma(f)+t^{2}\Gamma(\Delta f)+2t\Gamma(f,\Delta f)+o(t^{2}),$$
substituting it to the right side of the item (4),
\begin{equation}
\begin{split}
& \frac{1-e^{-2Kt}}{K}\Gamma(P_{t}f)+\frac{e^{-2Kt}-1+2Kt}{K^{2}n}\left(P_{t}\Delta f\right)^{2}\\
& =(2t-2Kt^{2}+o(t^{2}))\left(\Gamma(f)+t^{2}\Gamma(\Delta f)+2t\Gamma(f,\Delta f)+o(t^{2})\right)+\frac{2}{n}t^{2}(\Delta f)^{2}+o(t^{2})\\
& =2t\Gamma(f)+2t^{2}\left(2\Gamma(f,\Delta f)-K\Gamma(f)+\frac{1}{n}(\Delta f)^{2}\right)+o(t^{2}).
\end{split}
\notag
\end{equation}
combining with \eqref{equ:4}, with simply calculates, we will obtain the condition $CD(n,-K)$.
\end{proof}

Let $K\rightarrow 0$ or $n\rightarrow \infty$ in Theorem 3.1, we can obtain the following two corollaries.
\begin{corollary}
For any positive and bounded function $0<f\in \ell^{\infty}(V,\mu)$, the condition $CD(n,0)$ is equivalent to:
\begin{enumerate}
  \item $\Gamma(P_{t}f) \leq P_{t}(\Gamma(f))-\frac{2}{n}\int_{0}^{t}P_{s}\left(P_{t-s} \Delta f\right)^{2} ds.$
  \item $\Gamma(P_{t}f) \leq P_{t}(\Gamma(f))-\frac{2t}{n}\left(P_{t} \Delta f\right)^{2}.$
  \item $-2t\Gamma(P_{t}f)+\frac{2t^2}{n}\left(P_{t} \Delta f\right)^{2}\leq P_{t}f^{2}-(P_{t}f)^{2}\leq 2tP_{t}(\Gamma(f))-\frac{2t^2}{n}\left(P_{t} \Delta f\right)^{2}.$
\end{enumerate}
\end{corollary}

\begin{corollary}
For any positive and bounded function $0<f\in \ell^{\infty}(V,\mu)$, the condition $CD(\infty,-K)$ is equivalent to:
\begin{enumerate}
  \item $\Gamma(P_{t}f) \leq e^{2Kt}P_{t}(\Gamma(f)).$
  \item $\frac{1-e^{2Kt}}{K}\Gamma(P_{t}f)\leq P_{t}f^{2}-(P_{t}f)^{2}\leq \frac{e^{2Kt}-1}{K}P_{t}(\Gamma(f)).$
\end{enumerate}
\end{corollary}

\subsection{The equivalent properties of $CDE'$ condition}

In this subsection, we let the dimension of graph be $\infty$, and prove the gradient  of the graph.
\begin{theorem}
For any $K\in \mathbb{R}$, and any positive and bounded function $0<f\in \ell^{\infty}(V,\mu)$, the condition $CDE'(\infty,-K)$ is equivalent to:
$$\Gamma(\sqrt{P_{t}f}) \leq e^{2Kt}P_{t}(\Gamma(\sqrt{f})).$$
\end{theorem}
\begin{proof}
For any $0\leq s\leq t$, consider the functional
$$e^{2Ks}\psi(s)=e^{2Ks}P_{s}(\Gamma(\sqrt{P_{t-s} f})),\quad 0\leq s \leq t,$$
from Lemma 3.1, we have
$$\partial_{s}e^{2Ks}\psi(s)=2e^{2Ks}P_{s}(\widetilde{\Gamma}_{2}(\sqrt{P_{t-s} f})+K\Gamma(\sqrt{P_{t-s} f})).$$
Applying the condition $CDE'(\infty,-K)$ to the function $\sqrt{P_{t-s} f}$, we obtain
$$\widetilde{\Gamma}_{2}(\sqrt{P_{t-s} f})\geq -K\Gamma(\sqrt{P_{t-s} f}),$$
then for any $x\in V$, $\partial_{t}\psi(t)\geq 0$ is true. that is, the functional $e^{2Ks}\psi(s)$ is not decreasing with respect of $t$, so, $\psi(0)\leq\psi(t)$, that is
$$\Gamma(\sqrt{P_{t}f}) \leq e^{-2Kt}P_{t}(\Gamma(\sqrt{f})).$$

On the other side, we need prove the assertion that $\Gamma(\sqrt{P_{t}f}) \leq e^{2Kt}P_{t}(\Gamma(\sqrt{f}))$ implies the condition $CDE'(\infty,-K)$, as follows.

when $t=0$, the proof is trivial;

when $t\neq 0$,
\[
\begin{split}
0&\leq\lim_{t\rightarrow 0^{+}}\frac{1}{2t}[e^{2Kt}P_{t}(\Gamma(\sqrt{f}))-\Gamma(\sqrt{P_{t}f})]\\
&=\lim_{t\rightarrow 0^{+}}\frac{1}{2}[2Ke^{2Kt}P_{t}(\Gamma(\sqrt{f}))+e^{2Kt}\Delta P_{t}(\Gamma(\sqrt{f}))-2\Gamma(\sqrt{P_{t}f},\partial_{t}\sqrt{P_{t}f})]\\
&=\lim_{t\rightarrow 0^{+}}\frac{1}{2}[2Ke^{2Kt}P_{t}(\Gamma(\sqrt{f}))+e^{2Kt}\Delta P_{t}(\Gamma(\sqrt{f}))-2\Gamma(\sqrt{f},\frac{\Delta P_{t}f}{2\sqrt{P_{t}f}})]\\
 &=\frac{1}{2}\Delta(\Gamma(f))+K\Gamma(\sqrt{f})-\Gamma(\sqrt{f},\frac{\Delta f}{2\sqrt{f}})\\
 &=\widetilde{\Gamma}_{2}(\sqrt{f})+K\Gamma(\sqrt{f}),
\end{split}
\]
that is the condition $CDE'(\infty,K)$ is true.

This ends the proof.
\end{proof}

Yong Lin,\\
Department of Mathematics, Renmin University of China, Beijing, China\\
\textsf{linyong01@ruc.edu.cn}\\
Shuang Liu,\\
Department of Mathematics, Renmin University of China, Beijing, China\\
\textsf{cherrybu@ruc.edu.cn}
\end{document}